\documentclass[11pt]{article}
\usepackage{amssymb,amsmath}
\usepackage{amsthm}
\usepackage{amscd}
\newtheorem{theorem}{Theorem}[section]
\newtheorem{theorem*}{Theorem A\!\!}
\newtheorem{proposition}{Proposition}[section]
\newtheorem{proposition*}{Proposition A\!\!}

\newtheorem{corollary*}{Corollary A\!\!}

\newtheorem{lemma}{Lemma}[section]

\newtheorem{definition}{Definition}[section]

\DeclareMathOperator{\Mat}{Mat}

\DeclareMathOperator{\Id}{Id}

\DeclareMathOperator{\tr}{tr}

\DeclareMathOperator{\res}{res}
\DeclareMathOperator{\sgn}{sgn}
\DeclareMathOperator{\diag}{diag}

\begin{document}

\title{Symmetry breaking differential operators, the source operator  and Rodrigues formul\ae}

\author{Jean-Louis Clerc}

\date{ }
\maketitle

\abstract{A Rodrigues type formula is obtained for the symbols of the covariant bi-differential operators  on a simple real Jordan algebra.}

\section*{Introduction}
\emph{Symmetry breaking differential operators} (SBDO for short) are a classical notion in physics, and they have interested many authors in mathematics during the recent years. T. Kobayashi has designed a program to study the existence, the uniqueness and the construction of such operators; We will use a restricted version of the notion of SBDO, adapted to the present article and we will be concerned mostly by the constructive part of the program, even more precisely in an  effort to give explicit expressions for these operators. Here is a (non exhaustive) list of recent papers on the subject \cite{bc,bck,bck2,c1,c2,fos,j,k,kp16,ks,or,s}.

Let $M$ be a manifold, $G$ (called the big group) a Lie group acting on $M$, $N$ a submanifold of $M$ and $H$ (called the small group) a closed Lie subgroup of $G$ which preserves $N$. Let $\pi$ be a smooth representation of $G$ on $C^\infty(M)$ and $\rho$ a smooth representation of $H$ on $C^\infty(N)$. Let $D$ be a differential operator from $C^\infty(M)$ into $C^\infty(N)$. Then $D$ is said to be a \emph{symmetry breaking differential operator} (SBDO for short) if $D$ intertwines $\pi_{\vert H}$ and $\rho$, i.e.
\[\forall h\in H,\qquad  D\circ \pi(h) = \rho(h) \circ D\ .
\]

A specific situation is the tensor product situation, where the big group is $G\times G$ and the subgroup is the diagonal $diag(G) \simeq G$. The representation is a tensor product $\pi\otimes \pi'$ of two representations of $G$ and the symmetry breaking differential operator $D$ is also called a \emph{ covariant bi-differential operator}.

The present paper is a continuation of \cite{bck} written in collaboration with S. Ben Sa\"id and Kh. Koufany. The group $G$ is the conformal group of a simple real Jordan algebra $V$. The group $G$ is simple, and the subgroup $P$ of affine conformal transformations is a parabolic subgroup of $G$ with unipotent radical and such that its opposite parabolic $\overline P$ is conjugate to $P$. 

To  any $(\lambda, \epsilon)\in \mathbb C\times \{ \pm \}$ is associated a smooth representation $\widetilde \pi_{\lambda,\epsilon}$ induced by a character $\chi_{\lambda, \epsilon}$ of $P$ (degenerate non-unitary scalar principal series) realized on the sections of a line bundle over $X=G/\overline P$.

The action of $G$ on $X$ can be transferred as a rational action of $G$ on $V$, and 
the representation can be realized as $\pi_{\lambda, \epsilon}$ acting on smooth functions on $V$ (the \emph{noncompact picture}).

A typical example is $V= \Mat(r;\mathbb R)$, $G=SL(2r,\mathbb C)$, $X=Gr(r,2r;\mathbb R)$, and for $g^{-1} = \begin{pmatrix} a&b\\c&d\end{pmatrix}$
\[\pi_{\lambda, \epsilon}(g)  f(x) = \det(cx+d)^{-\lambda, \epsilon} f\big( (ax+b)(cx+d)^{-1}\big)\ .
\]
When $r=1$, the Jordan algebra is just $\mathbb R$, and the covariant bi-differential operators are the classical \emph{Rankin-Cohen brackets}. More details on these examples can be found in \cite{c1}.

Let $n$ be the dimension of $V$, $\det$ the determinant of the Jordan algebra and let $r$ be the rank of $V$. Consider the product $V\times V$ and let $diag(V) \simeq V$ be the diagonal of $V\times V$. Let \[\res: C^\infty(V\times V) \longrightarrow C^\infty(V)\] be the restriction map to the diagonal $\{(x,x),x\in V\} \simeq V$. The main result in this article is the following theorem, which is a new result even for the classical case of the Rankin-Cohen brackets.

\begin{theorem} Let $(\lambda,\epsilon), (\mu, \eta) \in \mathbb C\times \{ \pm \}$.
\smallskip

i) for any $s,t\in \mathbb C$ and for any $k\in \mathbb N$, there exists a polynomial $c^{(k)}_{\lambda, \mu}$ on $V\times V$ such that
\begin{equation}\label{defpol}
\det\left( \frac{ \partial}{\partial x} -\frac{ \partial}{\partial y} \right)^k (\det x)^{s+k}(\det y)^{t+k}= c^{(k)}_{s,t}(x,y)(\det x) ^s (\det y) ^t\ .
\end{equation}
Let  $B^{(k)}_{\lambda,\mu} : C^\infty(V\times V) \longrightarrow C^\infty(V)$ be the bi-differential operator given by
\[B^{(k)}_{\lambda,\mu}= \res\, \circ \, c^{(k)}_{\lambda-\frac{n}{r}, \mu-\frac{n}{r}}
\left( \frac{\partial}{\partial x}, \frac{\partial}{\partial y}\right)\ .\]
Then $B^{(k)}_{\lambda,\mu}$ satisfies the following covariance relation
\begin{equation}
\text{for any } g\in G,\quad
B^{(k)}_{\lambda,\mu} \circ \big(\pi_{\lambda, \epsilon}(g)\otimes \pi_{\mu, \eta}(g)\big)=
\pi_{\lambda+\mu+2k, \epsilon \eta}(g) \circ B^{(k)}_{\lambda,\mu}\ .
\end{equation}

\end{theorem}

The proof goes through four steps.
\vfill\eject

\bigskip

\centerline{ CONSTRUCTION OF THE SOURCE OPERATOR}
\bigskip

\centerline {CONSTRUCTION OF SBDO FROM THE SOURCE OPERATOR}
\bigskip

\centerline{RECURSION RELATION  FOR THE SYMBOLS OF THE SBDO}
\bigskip

\centerline{ RODRIGUES FORMULA FOR THE SYMBOLS OF THE SBDO}
\bigskip

The \emph{source operator} is a differential operator $E_{\lambda, \mu}$ with polynomial coefficients on $V\times V$ which has the following covariance property, valid for any $g\in G$ :
\begin{equation}
E_{\lambda \mu} \circ \big(\pi_{\lambda, \epsilon} (g) \otimes \pi_{\mu,\eta}(g) \big) = \big(\pi_{\lambda+1,- \epsilon} (g) \otimes \pi_{\mu+1,-\eta}(g)\big)\circ 
E_{\lambda, \mu}\ .
 \end{equation}
 The second step is an easy and  a context free consequence of the first step. The covariant bi-differential operators, going from $C^\infty(V\times V)$ into $C^\infty(V)$ are defined by
 \begin{equation}\label{defB}
 B^{(k)}_{\lambda, \mu} = \res\, \circ\, E_{\lambda+k-1, \mu+k-1}\circ \dots \circ E_{\lambda, \mu}\ ,
 \end{equation}
 and satisfy the covariance relation, valid for any $g\in G$
 \begin{equation}
 B_{\lambda,\mu}^{(k)} \circ \big(\pi_{\lambda, \epsilon}(g)\otimes \pi_{\mu, \eta}(g)\big) = \pi_{\lambda+\mu+2k,\epsilon \eta}(g) \circ  B_{\lambda,\mu}^{(k)}\ .
 \end{equation}
 Notice moreover that the operators $B^{(k)}_{\lambda, \mu}$ have constant coefficients, as can be deduced from the covariance property for the diagonal translations  $(x,y)\longmapsto (x+v, y+v)$.
 
 The two first steps have been achieved in many different geometric  situations, and also for representations acting on sections of vector bundles (\cite{bc, c1, c2, bck, bck2, fos}). Let us mention that the idea of the source operator is reminiscent of the $\Omega$ process used for the construction of the \emph{transvectants} (see Section 5 in \cite{c1}).
 
For the present geometric situation, the first step was achieved in \cite{bck}, mainly by using delicate Fourier analysis on $V$.  In that sense, the present work is a continuation of \cite{bck},  in an effort to find a more explicit formula for the covariant bi-differential operators. 
  
For the third step, the new idea in the present work is to deduce from the definition \eqref{defB} a recurrence relation between the symbol of $B_{\lambda,\mu}^{(k)}$ and the symbol of $B_{\lambda+1,\mu+1}^{(k-1)}$. In the fourth step, this recurrence relation is reminiscent from the recurrence relation that can be deduced from the classical \emph{Rodrigues formul\ae} and this observation suggests a solution for the relation recurrence. 
 
The r\^ole of the Rodrigues formula was discovered in three (rather) elementary cases by explicit calculations :  the Rankin-Cohen brackets,  the Juhl operators and the  conformally covariant bi-differential operators on $\mathbb R^n$. These results are presented at the end of the paper.  For the two first cases, the relation of the symbols of the operators to families of orthogonal polynomials (the Jacobi polynomials for the Rankin-Cohen brackets, the Gegenbauer polynomials for the Juhl operators) had been observed before (see \cite{ks}, \cite{kp16}, \cite{j}). In both cases, the recurrence relation obtained for the symbols of the differential operators could be compared to the classical Rodrigues formula for the families of orthogonal polynomials involved. The third case was new, but no connection to families of orthogonal polynomials or special functions is known.  
 
The Fourier transform plays an important r\^ole in \cite{bck}. In order to gain flexibility in some computations, typically for rewriting a differential operator with polynomial coefficients in its \emph{normal form}, that is to say differentiation before multiplication, {\it ad hoc} symbolic calculi, both on the space and on its dual, are developped in Section 1 on a general real vector space. They are inspired by the symbolic calculus in the Weyl algebra and/or the pseudo-differential symbolic calculus. However, it is more an  algebraic calculus as no asymptotic analysis is needed.

In Section 2, the construction of the source operator is recalled, mostly relying on \cite{bck}. The calculi developed in Section 2 are applied to the objects introduced in \cite{bck}, and 

Section 3 introduces the covariant bi-differential operators and the recurrence relation satisfied by their symbols.
 
Section 4 is devoted to the Rodrigues formula.

Section 5 presents the three mentioned examples.

The definition \eqref{defpol} of the polynomials $c^{(k)}_{\lambda, \mu}$ through a Rodrigues type formula arises the question of wether they could be associated to a theory of orthogonal polynomials, at least for particular values of the parameters $\lambda, \mu$. It is only speculation at the moment, but the question deserves further investigation.

\section{Dual symbolic calculi on $E$ and $E^*$}

\subsection{The symbolic calculus on $E$}

Let $E$ be a real vector space of dimension $N$. Choosing a basis of $E$ allows to identify $E$ with  $\mathbb R^N$. The corresponding coordinates of an element $x\in E$ will be denoted by $(x_1,x_2,\dots, x_N)$.  For $1\leq j\leq N$, let
$ \displaystyle  \partial_j = \frac{\partial}{\partial x_j}$.
For $\alpha=(\alpha_1,\alpha_2,\dots, \alpha_N)$ a $N$-multiindex, we let
\[\vert \alpha\vert = \alpha_1+\alpha_2+\dots \alpha_N,\qquad  \alpha! = \alpha_1!\alpha_2!\dots \alpha_N!\]
\[x^\alpha = x_1^{\alpha_1}x_{\alpha_2} \dots x_N^{\alpha_N},\qquad \partial^\alpha = \partial_1^{\alpha_1} \partial_2^{\alpha_1}\dots \partial_N^{\alpha_N}\ .
\]

Let $E^*$ be the dual of $E$, also identified with $\mathbb R^N$ via the basis of $E^*$ dual to the chosen basis for $E$. Elements of $E^*$ will be denoted by greek letters and their coordinates by
$\xi = (\xi_1,\xi_2,\dots, \xi_N)$. The duality between $E$ and $E^*$ is denoted by $x.\xi, x\in E,\xi \in E^*$. The dual ${(E^*)}^*$ is identified with $E$.

For $1\leq j\leq N$, let $\displaystyle \partial_{*j}= \frac{\partial}{\partial \xi_j}$, and as before, for a N-multiindex $\alpha$, let
\[\xi^\alpha = \xi_1^{\alpha_1} \xi_2^{\alpha_2}\dots x_N^{\alpha_N} ,\qquad  \partial_*^{\alpha } = \partial_{*1}^{\alpha_1}\partial_{*2}^{\alpha_2}\dots \partial_{*N}^{\alpha_N}
\]
The Fourier transform is viewed as map from functions on $E$ to functions on $E^*$, given by 
\[\mathcal F(f)(\xi) = \widehat f (\xi) = \int_ E e^{-i x.\xi} f(x)\, dx\ .
\]
On $E^*$, the inverse Fourier transform is defined by
\[(\mathcal F^{\,*}g)(x)  = \check g (x) = \left(2\pi\right)^{-N} \int_{E^*} e^{ix.\xi}g (\xi )\, d\xi.
\]
The classical formula of Fourier analysis are
\begin{equation}
\begin{split}
\mathcal F \circ x^\alpha = ({i\partial_*})^\alpha \circ \mathcal F,& \qquad \mathcal F \circ { (\frac{1}{i} \partial)^\alpha} = \xi^\alpha \circ \mathcal F
\\ 
\mathcal F^* \circ \xi^\alpha = (\frac{1}{i} \partial)^\alpha \circ \mathcal F^*,& \qquad \mathcal F^* \circ ({i\partial_*})^\alpha = x^\alpha \circ \mathcal F^*\ .
\end{split}
\end{equation}

Define  $Op(E)$ to be the space of operators from $\mathcal S(E)$ into $\mathcal S'(E)$ generated by finite combination of products of a convolution operator by a tempered distribution on $E$ followed by a multiplication by a polynomial function on $E$. For $k$  a tempered distribution, denote by $K$ the corresponding convolution operator, viewed as an operator from $\mathcal S(\mathbb R^N)$ into $\mathcal S'(\mathbb R^N)$ given by
\[K\varphi (x) = (k\star \varphi)(x)\ .
\]
 For $p$ is a polynomial function on $E$, the multiplication operator
 \[f\longmapsto (pf)(x) =p(x) f(x)
 \]
 seen as an operator from $\mathcal S(E)$ into $\mathcal S(E)$ or from $\mathcal S'(E)$ into $\mathcal S'(E)$ is denoted by by $p(x)$ or $p$ depending on the context.
 
Hence an element of  $Op(E)$ is a finite linear combination of operators of the form $p(x)\circ K$, respecting the convention of normal order (convolution first, then multiplication). Any element of $Op(E)$ can be written in a unique way as
\[\sum_{\alpha} x^\alpha L_\alpha
\]
where for each N-multiindex $\alpha$ $L_\alpha$ is a convolution operator with a tempered distribution $E$, and such that $L_\alpha$ is $0$ except for a finite number of $\alpha$'s.

The space $Op(E)$ contains the Weyl algebra $\mathcal W(E)$ (algebra of differential operators with polynomial coefficents). In fact, if $D$ is a constant coefficients differential operator, then for a test function $f\in \mathcal S(E)$,
\[Df = D(\delta_0\star f) = (D\delta_0) \star f,
\]
where $\delta_0$ is the Dirac distribution at $0$. As $D\delta_0$ is a tempered distribution, $D$ belongs to $Op(E)$. Moreover, if $p$ is a polynomial function on $E$, the operator $p(x) D$ also belong to $Op(E)$ and hence the Weyl algebra $\mathcal W(E)$ is included in $Op(E)$.

The \emph {symbol}  is the Fourier side version of these operators. For $K$ the convolution operator with the tempered distribution $k(x)$, the symbol is defined as the distribution on $E^*$ equal to $\widehat k(\xi)$, whereas for $p$ a polynomial function on $E$, then its symbol is just $p(x)$. For the product in normal order $p(x)K$, the symbol is defined to be $p(x) \widehat k(\xi)$. This corresponds to the following formula (to be understood as an equality of tempered distributions) which is easy to prove :
\[p(x)Kf\,(x) = (2\pi)^{-N} \int_E e^{ix.\xi} p(x) \widehat k(\xi) \widehat f(\xi) d\xi\ .
\]
\begin{definition} Let $L=\sum_\alpha x^\alpha K_\alpha$ be an element of $Op(E)$. Its symbol $symb(L)$ is the element of $\mathcal Pol(E, \mathcal S'(E^*))$ defined by
\[symb(L)(x,\xi) = \sum_{\alpha} x^\alpha \mathcal Fk_\alpha(\xi)\ .
\]
Conversely, given $s(x,\xi)= \sum_\alpha x^\alpha l_\alpha(\xi)\in \mathcal Pol(E, \mathcal S'(E^*))$, the operator
\[L = \sum_\alpha x^\alpha K_\alpha 
\]
where $K_\alpha$ is the convolution operator by the tempered distribution $\mathcal F^*l_\alpha(x)$
belongs to $Op(E)$ and its symbol is equal to $s(x,\xi)$.
\end{definition}

Let $D\in \mathcal W(E)$. Recall that its symbol (in the usual sense) $\sigma(D)$ is the polynomial function on $E\times E^*$ defined by
\[D (e^{ix.\xi}) = \sigma(D)(x,\xi) e^{ix.\xi}
\]
\begin{proposition} Let $D\in \mathcal W(E)$. Then  
\[symb(D) =\sigma(D)\ .
\]
\end{proposition}
\begin{proof} Let first $D= \partial^\alpha$ for some N-multiindex $\alpha$. Then $\sigma(D) (x,\xi) = (i\xi)^\alpha$. On the other hand, $D$ is the convolution operator with the distribution $ D^\alpha \delta_0$ and $\mathcal F (\partial^\alpha \delta_0) (\xi) = (i\xi)^\alpha\mathcal F \delta_0 (\xi) = (i\xi)^\alpha$, so that $symb(\partial^\alpha)(\xi) = (i\xi)^\alpha= \sigma(\partial^\alpha)(\xi)$. The general case follows easily. 
\end{proof}

The space $Op(E)$ is not an algebra, but some compositions are possible, and then the symbol of the composed operator is given by a formula similar to the composition formula in the pseudo-differential calculus or the composition formula for the symbols in the Weyl algebra.

Let $L\in Op(E)$ and $D\in\mathcal W(E)$. Then $L\circ D$ is well defined as an operator from $\mathcal S(E)$ into $\mathcal S'(E)$.

\begin{proposition}\label{compKp}
Let $L$ be an element of $Op(E)$, and let $D\in \mathcal W(E)$. Then $L\circ D$ belongs to $Op(E)$ and its symbol is given by
\begin{equation}\label{compsymb}
symb(L\circ D) : =symb(L)\sharp symb(D) = \sum_\alpha \frac{1}{\alpha!}\, \left(\frac{1}{i}\partial_\xi\right)^\alpha  symb(L)\,\partial_x^\alpha symb(D)\,\ .
\end{equation}
\end{proposition}
\begin{proof}
First assume that $L=K$ is a convolution operator by a tempered distribution $k$, and $D$ is the multiplicateionn operator by $p(x)$, where $p$ is a polynomial function on $E$. Then
\[(K\circ p)\varphi(x) = \int_{\mathbb R^N} k(x-y) p(y) \varphi(y) dy
\]
\[=(2\pi)^{-N} \int_{\mathbb R^N}\int_{\mathbb R^N}e^{i(x-y).\xi}\, \widehat k(\xi)\, p(y) \varphi(y)\, dy
\]
which by the change of variable $z=y-x$ becomes
\[(2\pi)^{-N} \int_{\mathbb R^N}\int_{\mathbb R^N}e^{-iz.\xi}\, \widehat k(\xi)\, p(x+z) \varphi(x+z)\  dz \, d\xi\ .
\]
Use the (exact) Taylor expansion for the polynomial $p$ at $x$ to get
\[= (2\pi)^{-N} \sum_\alpha \frac{1}{\alpha!}\partial_x^\alpha p\, (x)\int_{\mathbb R^N}\int_{\mathbb R^N}e^{-iz.\xi} \widehat k(\xi) z^\alpha \varphi(x+z)\, dz \,d\xi\ .
\]
Now use first
\[z^\alpha\int_{\mathbb R^N}e^{-i\xi.z} \widehat k(\xi) d\xi = \int e^{-iz.\xi} \left(\frac{1}{i} \partial_\xi\right)^\alpha \widehat k (\xi)\, d\xi
\]
to get
\[=(2\pi)^{-N} \sum_\alpha \frac{1}{\alpha!}\partial_x^\alpha p\, (x)\iint e^{-iz.\xi}{\left(\frac{1}{i} \partial_\xi\right)}^\alpha \widehat k (\xi)\varphi(x+z)\, dz\,d\xi.
\]
Now use
\[\int_{\mathbb R^N} e^{-iz.\xi} \varphi(x+z) \,dz = e^{ix.\xi} \,\widehat \varphi (\xi)
\]
to get
\[=(2\pi)^{-N} \sum_\alpha \frac{1}{\alpha!}\partial_x^\alpha p\, (x)\int e^{ix.\xi} {\left(\frac{1}{i} \partial^*\right)}^\alpha\widehat k (\xi)\widehat \varphi (\xi) d\xi,
\]
and the conclusion follows. The general case is an easy consequence of this special case,  as multiplication on the left by a polynomial function (resp. composition on the right with a constant coefficients differential operator) corresponds for the symbols to a multiplication by $p(x)$ (resp. to a multiplication by $\sigma(D)(\xi)$).

\end{proof}

In particular, the composition law for the symbols of two such differential operators coincides with the  classical formula in the Weyl algebra
\begin{equation}symb(D\circ F)  = d(x,\xi)\, \sharp\, f(x,\xi) = \sum_{\alpha} \frac{1}{\alpha!}\,{\left(\frac{1}{i} \partial_\xi\right)}^\alpha d(x,\xi)\, \partial^\alpha_x f(x,\xi) \ .\end{equation}

\subsection{The dual symbolic calculus on $E^*$}

The symbolic calculus has its dual (via the Fourier transform) version on $E^*$. 
The family $Op(E^*)$ is the family of finite combinations of multiplications by tempered distributions composed with constant coefficients differential operators on $E^*$, viewed as operators from $\mathcal S(E^*)$ into $\mathcal S'(E^*)$. The family $Op(E^*)$ is obtained from $Op(E)$ by transmutation by the Fourier transform. Any element of $Op(E^*)$ can be written in a unique way as
\[L = \sum_\alpha l_\alpha(\xi) {\partial_\xi}^\alpha
\]
where for each N-multiindex $\alpha$ $l_\alpha$ is a tempered distribution and $l_\alpha = 0$ but for a finite number of $\alpha$'s.

By definition, the symbol of the multiplication operator by the the tempered distribution $k$ on $E^*$ is $k(\xi)$. On the other hand, if $D$ is  a differential operator on $E^*$ with polynomial coefficients, its (usual) symbol is defined by is the polynomial function $\sigma^*(D)$ on $E^*\times E$ defined by 
\[D\, (e^{i\xi.x}) = \sigma^*(D)(\xi, x) \, e^{i \xi.x}\ .
\]
Notice that this definition implies that
\[\sigma^*{\left( \frac{1}{i}\partial^*_j\right)} (\xi, x) =  x_j\ .
\]

\begin{definition} Let $L = \sum_{\alpha} l_\alpha \partial_\xi^\alpha$. Its symbol is defined by
\[symb^*(L) (x,\xi) = l_\alpha(\xi) (ix)^\alpha
\]
\end{definition}
The space $Op(E^*)$ contains the Weyl algebra $\mathcal W(E^*)$ of differential operators on $E^*$ with polynomial coefficients and for $D \in \mathcal W(E)$, $symb^*(D)$
 coincides with its usual symbol $\sigma^*(D)$.
 
 The space $Op(E^*)$ is not a algebra, but some compositions are possible, and then the symbol of the composed operator is given by a formula similar to the composition formula already seen for $Op(E)$. If $L\in Op(E^*)$ and $D\in \mathcal W(E^*)$, then 
 $D\circ L$ is well defined as an operator from $\mathcal S(E^*)$ into $\mathcal S'(E^*)$.
 
\begin{proposition}\label{compDL}
 Let $L\in Op(E^*)$ and let $D\in \mathcal W(E^*)$.  Then
$D\circ L$ belongs to $Op(E^*)$ ad its symbol is given by 
\begin{equation}\label{symb*}
symb^*(D\circ L) := symb^*(D)\,\flat\,symb^*(L) = \sum_{\alpha} \frac{1}{\alpha!}\, \partial^\alpha_x symb^*( D) (\frac{1}{i}\partial_\xi)^\alpha symb^*(L)\,\ .
\end{equation}
\end{proposition}

\begin{proof} First assume that $L$ is a multiplication operator by a tempered distribution $k(\xi)$ and $D$ is  a constant coefficient differential operator, with symbol equal to a polynomial $p(x)$.
As both sides of \eqref{symb*} are linear in $p$, we may even assume that $p$ is a monomial, say $p(x) = x^\beta$ for some multi-index $\alpha$. As $\big(p(x)q(x)\big)\,\flat\, k(\xi) = p(x)\,\flat\,\big(q(x)\,\flat\,k(\xi)\big)$ for 
$p,q\in \mathcal P(E)$, it suffices to prove the formula for $p(x) = x_j^{\beta_j}$, or equivalently for $D=({\frac{1}{i}\partial_j^*})^{\beta_j}$. In this case, let $\varphi\in \mathcal S(E^*)$. Then
\[
({i\partial_j^*})^{\beta_j} (k(\xi) \varphi(\xi))=
\sum_{\alpha_j=0}^{\beta_j} \begin{pmatrix}\beta_j\\ \alpha_j \end{pmatrix} ({\frac{1}{i}\partial_j^*})^{\alpha_j} k(\xi) ({\frac{1}{i}\partial_j^*})^{\beta_j-\alpha_j}\varphi(\xi)\ .
\]
Now use
\[ symb^*\big({(\frac{1}{i}\partial_j^*)}^{\beta_j-\alpha_j}\big) =x_j^{\beta_j-\alpha_j} = \frac{(\beta_j-\alpha_j)!}{\beta_j!}\]

to get
\[symb^*(x^{\beta_j}\circ k(\xi)) = x^{\beta_j} \, \flat\, k(\xi) = \sum_{\alpha_j} \frac{1}{\alpha_j!}({\frac{1}{i}\partial_j^*})^{\alpha_j} k(\xi)\partial_j^{\alpha_j} x_j^{\beta_j}\ .
\]
This proves \eqref{symb*} for the case where $L$ is a multiplication operator by a tempered distribution and $D$ is a constant coefficient differential operator. The general case follows easily.
\end{proof}
Again, the space $Op(E^*)$ contains the space of differential operators with polynomial coefficients on $E^*$ the Weyl algebra $\mathcal W(E^*)$, and the symbol of such an operator coincides with its usual definition. Also the composition formula for the symbol of two differential operators with polynomials coefficients is given by 
\[c(\xi,x) \,\flat\, d(\xi,x) = \sum_\alpha\frac{1}{\alpha!}\, \partial_x^\alpha c(\xi,x) \, ({\frac{1}{i}\partial_\xi})^\alpha d(\xi,x)\ .
\]
\begin{proposition}\label{x=0}
Let $D$ be a differential operator on $E^*$ with polynomial coefficients, let $d^*(\xi,x)$ be its symbol and let $p$ be a polynomial function on $E^*$. Let $c^*(\xi,x) =d^*(\xi,x)\,\flat\, p(\xi) $.
Then
\[c(\xi, 0) = Dp(\xi)\ .\]
\end{proposition}
\begin{proof}
By linearity, it is sufficient to prove the result when $D= \xi^\gamma (\frac{1}{i}\partial_\xi)^\delta$ for $\gamma, \delta$ arbitrary $N$-multiindices. Then  $d^*(\xi,x) = \xi^\gamma x^\delta$ so that
\[c(\xi,0) = \big(\xi^\gamma x^\delta\,\flat\, p(\xi)\big)(\xi,0)  = \sum_{\alpha!} \frac{1}{\alpha} \xi^\gamma \partial_x^\alpha (x^\delta)(0) (\frac{1}{i}\partial_\xi)^\alpha p(\xi)\ .
\]
But $\partial_x^\alpha (x^\delta)(0)=0$ unless $\alpha = \delta$, in which case
\[\partial_x^\delta (x^\delta)(0) = \delta! 
\]
so that
\[c(\xi,0)  = \xi^\gamma (\frac{1}{i}\partial_\xi)^\delta p(\xi) = Dp(\xi)\ .
\]
\end{proof}

\begin{proposition}\label{compduality}
 Let $e(x,\xi)$ be a symbol on $V\times V^*$ and $d(x,\xi)$ be a polynomial function on $E\times E^*$. Let
\[e^*(\xi,x) = e(x,\xi),\qquad d^*(\xi, x) = d(x,\xi)\ .
\] 
Then
\begin{equation}\label{duality}
e(x,\xi)\, \sharp \,d(x,\xi) = d^*(\xi, x) \,\flat \,e^*(\xi, x) \ .
 \end{equation}
 The same result  is valid when $e(x,\xi)$ is a polynomial and $d(x,\xi)$ is a symbol.
\end{proposition}
\begin{proof}
The identity  \eqref{duality} is a consequence of the two composition formulas for the symbols \eqref{compsymb} and \eqref{symb*}.
\end{proof}
\section{The source operator on a real simple Jordan algebra}

Let $V$ be a simple real Jordan algebra, of dimension $n$ and rank $r$. Recall that the \emph{conformal group} of $V$ is the group generated by the translations, the structure group of $V$ and the inversion. In \cite{bck} was introduced  a group $G$ which is locally isomorphic to the conformal group (a twofold covering of the proper conformal group, see details in Section 7 of \cite{bck}). Let $P$ be the parabolic subgroup of $G$ corresponding to the conformal affine maps (generated by the translations and the structure group). The unipotent radical $N$ of $P$ can be identified with $V$. Let $\overline P$ be the opposite parabolic subgroup to $P$. Then $G/\overline P$ is the \emph{conformal completion} of $V$ and the map $V\simeq N \longrightarrow G/\overline P$ gives the imbedding of $V$ in its completion.

The characters of $P$ are described by $(\lambda, \epsilon)\in \mathbb C\times \{\pm\}$, and for each $(\lambda, \epsilon)$ there is an induced representation (belonging to the degenerate non-unitary scalar principal series) $\widetilde \pi_{\lambda, \epsilon}$ of $G$ acting on the space $\mathcal H_{(\lambda,\epsilon)}$ of smooth sections of a line bundle over $X$. The \emph{Knapp-Stein operators} are a meromorphic (w.r.t. $\lambda$) family of operators 
\[\widetilde I_{\lambda, \epsilon} : \mathcal H_{\lambda, \epsilon} \longrightarrow \mathcal H_{\frac{2n}{r} -\lambda, \epsilon} 
\]
of intertwining operators, i.e. they satisfy for any $g\in G$
\begin{equation}
\widetilde I_{\lambda, \epsilon} \circ \widetilde\pi_{\lambda, \epsilon}(g) =\widetilde \pi_{\frac{2n}{r}-\lambda, \epsilon}(g)\circ \widetilde I_{\lambda,\epsilon} \ .
\end{equation}

There is another realization of the representations $\pi_{\lambda, \epsilon}$, called the \emph{ noncompact picture}, obtained by using the embedding $V\longrightarrow X$. The action of $G$ on $V$ is a rational (not everywhere defined) action, the space considered for the representation is the space of smooth functions $C^\infty(V)$ and the action can be explicitly written in this  ralization.

Recall the following notation : for $s\in \mathbb C, \epsilon\in \{\pm\}$ and for $t\in \mathbb R^*$,
\[t^{s,\epsilon} = \left\{\begin{matrix} \vert t\vert^s &\quad \text{if } \epsilon = +\\
\sgn(t) \vert t\vert^s&\quad \text{if } \epsilon = -
\end{matrix}\right.\qquad .
\]
Keeping same notation as above, the representation $\pi_{\lambda, \epsilon}$ is given by
\begin{equation}
\pi_{\lambda, \epsilon}(g) f (x) = a(g^{-1},x)^{-\lambda, \epsilon} f\big(g^{-1}(x)\big)
\end{equation}
where $a$ is a smooth cocycle on $G\times V$, and $a\big(g^{-1},x\big)=0$ precisely when $g^{-1}$ is not defined at $x$.

The Knapp-Stein intertwining operators can also be transferred and they are given (up to a scalar) by
\begin{equation}\label{defKS}
I_{\lambda, \epsilon} f(x) = \int_V \det(x-y)^{-\frac{2n}{r} +\lambda, \epsilon} f(y) dy\ ,
\end{equation}
where $\det$ is the \emph{determinant} polynomial of $V$.

The Knapp-Stein operators are convolution operators, and the family $\det(x)^{s\epsilon}$ which is well defined for $\Re(s) >>0$ is extended by analytic continuation to a meromorphic family of tempered distributions (see more details in \cite{bck}). The following result plays an important r\^ole in the sequel. 
\begin{proposition} The Fourier  transform of the kernel of the intertwining operator $k_{\lambda, \epsilon}(\xi)$  is given by
 \begin{equation}\label{Fintw}
 \begin{split}
  k_{\lambda, \epsilon}(\xi)& =\mathcal F\left( \det x^{-\frac{2n}{r}+\lambda, \epsilon}\right)(\xi) \\ &=\kappa_+(\lambda, \epsilon) \det \xi^{\frac{n}{r}-\lambda,+} +\kappa_-(\lambda, \epsilon) \det \zeta^{\frac{n}{r}-\lambda, -},
  \end{split}
 \end{equation}
 where $\kappa_\pm(\lambda)$ are meromorphic functions  on $\mathbb C$, not both identically $0$.
 \end{proposition}
 This result was obtained in \cite{bck} Section 5, where the explicit values of $\kappa_\pm(\lambda, \epsilon)$ can be found.

Consider now $(\lambda, \epsilon), (\mu,\eta)\in \mathbb C\times \{ \pm \}$ and form the tensor product $\widetilde \pi_{\lambda, \epsilon}\otimes \widetilde \pi_{\mu, \eta}$. In the compact realization, the  space of the representation (after completion)  is $\mathcal H_{(\lambda, \epsilon),(\mu,\eta)}$, the space of smooth sections of a line bundle over $X\times X$. Let $\pi_{\lambda, \epsilon}\otimes \pi_{\mu, \eta}$ be the corresponding tensor product representation in the non-compact picture, actin on  the space $C^\infty(V\times V)$.

Denote by $M$ the operator on $C^\infty(V\times V)$ given by
\begin{equation}
f\in C^\infty(V\times V) ,\qquad Mf(x,y) = \det(x-y) f(x,y)\ .
\end{equation}
The operator $M$ satisfies a covariance relation,
\[M\circ \big(\pi_{\lambda, \epsilon}(g)\otimes \pi_{\mu, \eta}(g)\big)= \big(\pi_{\lambda-1, -\epsilon}(g)\otimes \pi_{\mu-1, -\eta}(g)\big)\circ M\ .
\]
Among other things, this relation shows that $M$ can be lifted to the compact picture model, namely as an operator $\widetilde M$
\[\widetilde M : \mathcal H_{(\lambda, \epsilon), (\mu, \eta)}\longrightarrow \mathcal H_{(\lambda-1,-\epsilon), (\mu-1,-\eta)}\ ,\]
and intertwining $\widetilde \pi_{\lambda, \epsilon} \otimes \widetilde \pi_{\mu, \eta}$ and $\widetilde \pi_{\lambda-1,-\epsilon} \otimes \pi_{\mu-1,-\eta}$
Strictly speaking, the operator now depends on  the parameters of the representations, but its local expression in the noncompact picture does not.

On the compact realization, now define
\begin{equation}\label{defF}
\widetilde F_{(\lambda, \epsilon), (\mu, \eta)}= \left(\widetilde I_{\frac{2n}{r}-1-\lambda, \epsilon} \otimes \widetilde I_{\frac{2n}{r}-1-\mu, \eta}\right) \circ \widetilde M\circ \left(
\widetilde I_{\lambda, \epsilon} \otimes \widetilde I_{\mu, \eta}\right)\ .
\end{equation}

The main result of \cite{bck} is
\begin{theorem} The operators $\widetilde F_{(\lambda,\epsilon),(\mu, \eta)}$ form a meromorphic (in $(\lambda, \mu)$) family of \emph{differential} operators on $X\times X$, which satisfy the covariance relation, valid for any $g\in G$
\begin{equation}
\widetilde F_{(\lambda,\epsilon),(\mu, \eta)}\circ \left(\widetilde \pi_{\lambda, \epsilon}(g)\otimes \widetilde\pi_{\mu, \eta(g)}\right)=  \left(\widetilde\pi_{\lambda+1, -\epsilon}(g)\otimes \widetilde\pi_{\mu+1, -\eta}(g)\right)\circ \widetilde F_{(\lambda,\epsilon),(\mu, \eta} \ .
\end{equation}
\end{theorem}
Notice that the covariance property is an obvious consequence of the definition, whereas the proof that it is a differential operator requires much work. 

The operator $\widetilde F_{(\lambda, \epsilon),(\mu, \eta)}$ has a local expression in the noncompact picture, which turns out to be independent of the signs $\epsilon, \eta$, and is a differential operator on $V\times V$ with polynomial coefficients, henceforth denoted by $F_{\lambda, \mu}$. In the actual proof in \cite{bck}, the operator $F_{\lambda,\mu}$ is constructed in the noncompact picture and shown to be covariant, which guarantees its lifting to the compact picture. 

Notice however that it is not possible to use the analog of the composition formula \eqref{defF} directly on $V\times V$ to define $F_{\lambda, \mu}$, as it is \emph{not} possible to find a reasonable functional space on $V\times V$ which would be stable by the three operators concerned. However as $M$ preserves both $\mathcal S(V\times V)$ and $\mathcal S'(V\times V)$, it is at least possible to compose $M$ and one Knapp-Stein operator (in any order), obtaining an operator mapping $\mathcal S(V\times V) $ into $\mathcal S'(V\times V)$. It is then possible to apply the symbolic calculus developed in the previous section and to gain some information on the operator $F_{\lambda, \mu}$.

The distinction between $V$ and its dual $V^*$ will be explicitly made, whereas \cite{bck} used the non degenerate bilinear form $\tau(x,y) = \tr(xy)$ to identify $V$ and $V^*$.

\begin{proposition} The operators $\widetilde I_{\lambda, \epsilon}$ satisfy 
\begin{equation} 
\widetilde I_{\lambda, \epsilon} \circ \widetilde I_{\frac{2n}{r}-\lambda, \epsilon} = \kappa_2(\lambda, \epsilon) \Id\ ,
\end{equation}
where $\kappa_2(\lambda, \epsilon)$ is a meromorphic function $\not \equiv 0$ on $\mathbb C$. 
\end{proposition}

This result is well known, and can be shown through the noncompact realization, first when $\Re(\lambda) = \frac{n}{r}$ in which case the representation is unitary and the Knapp-Stein operator is (up to a scalar)
is also unitary on $L^2(V)$, as can be seen on  the non compact picture, using the Fourier transform of the kernel of $I_{\lambda, \epsilon}$. The inverse of $\widetilde I_{\lambda, \epsilon}$ is equal (up to a scalar) to $\widetilde I_{\frac{2n}{r}-\lambda, \epsilon} $ which again can be seen on the noncompact picture, after a Fourier transform. The relation is then extended by analytic continuation to  all generic $\lambda$ (i.e. outside of the poles of the Knapp-Stein operators).

From the definition of $\widetilde F_{(\lambda, \epsilon), (\mu, \eta)}$ it is possible to deduce the two following equalities.
\begin{equation}\label{twoterms1}
\widetilde F_{(\lambda, \epsilon), (\mu, \eta)} \circ \left(\widetilde I_{\frac{2n}{r}-\lambda,\epsilon}\otimes \widetilde I_{\frac{2n}{r}-\mu,\eta}\right)= \kappa_3(\lambda, \epsilon; \mu, \eta) \left(\widetilde I_{\frac{2n}{r}-1-\lambda,-\epsilon}\otimes \widetilde I_{\frac{2n}{r}-1-\mu,-\eta}\right)\circ \widetilde M
\end{equation}
\begin{equation}\label{twoterms2}
\left(\widetilde I_{\lambda+1,-\epsilon}\otimes \widetilde I_{\mu+1,-\eta}\right) \circ \widetilde F_{(\lambda, \epsilon), (\mu, \eta)} =  \kappa_4(\lambda, \epsilon; \mu, \eta) \widetilde M\circ \left(\widetilde I_{\lambda,\epsilon}\otimes \widetilde I_{\mu,\eta}\right)
\end{equation}
where $\kappa_3$ and $\kappa_4$ are meromorphic $\not \equiv 0$ functions on $\mathbb C\times \mathbb C$.

Now transfer these identities to the noncompact picture. As $ I_{\lambda, \epsilon}$  is a convolution operator, let $k_\lambda$ be its Fourier transform, which is at the same time its symbol. As the operator $F_{\lambda, \mu}$ is a differential operator 
with polynomial coefficients, let $f_{\lambda, \mu}(x,y,\xi, \zeta)$ be its symbol.
The two identities \eqref{twoterms1} and \eqref{twoterms2} can be translated using the symbolic calculus on $E=V\times V$.
\begin{proposition}
\begin{equation}\label{flambdamu1}
\begin{split}
f_{\lambda, \mu}(x,y, \xi, \zeta) \left(k_{\frac{2n}{r} -\lambda, \epsilon}(\xi)\otimes k_{\frac{2n}{r}-\mu, \eta}(\zeta)\right)
\\
 = \kappa_3(\lambda, \epsilon; \mu, \eta)\, \left(k_{\frac{2n}{r}-1 -\lambda,- \epsilon}(\xi)\otimes k_{\frac{2n}{r}-1-\mu, -\eta}(\zeta)\right) \,\sharp \,\det(x-y) 
\end{split}
\end{equation}
\begin{equation}\label{flambdamu2}
\begin{split}
\left(k_{\lambda+1,-\epsilon} (\xi) \otimes k_{\mu+1-\eta}(\zeta)\right)\, \sharp \,f_{\lambda, \mu} (x,y,\xi, \zeta) \\
= \kappa_4(\lambda, \epsilon; \mu, \eta) \det(x-y)\, \left( k_{\lambda,\epsilon} (\xi) \otimes k_{\mu,\eta}(\zeta)\right)\ .
\end{split}
\end{equation}
\end{proposition}

In \cite{bck},Theorem 4.8 was the key result in the construction of the operator $F_{\lambda, \mu}$. Here it is reinterpreted in order to fit with the symbolic calculus developed  in the previous section. Let $\det^*$ the polynomial on $E^*$ obtained through the identification of $V^*$ with  $V$ through the bilinear form $\tau$, and let
\[{V^*}^\times= \{\xi\in V^*, {\det}^*(\xi) \neq 0\} \ .
\]
\begin{proposition}\label{defDst}
For generic $(s,\epsilon), (t, \eta)\in \mathbb C\times \{ \pm \}$ there exists a differential operator $D_{s,t}$ with polynomial coefficients on $V^* \times V^*$ such that
\[\det\left(\frac{1}{i}( \partial_\xi-\partial_\zeta) \right)\circ  \left({\det}^* \xi ^{s,\epsilon} {\det}^* \zeta ^{t,\eta} \right)= \left({\det}^* \xi^{s-1,-\epsilon}{\det}^* \zeta^{t-1, -\eta}\right)\circ D_{s,t} 
\]
\end{proposition}
\begin{proof} The existence of such an operator is proved in \cite{bck}, but only on ${V^*}^\times \times {V^*}^\times$. To complete the proof of the present proposition, assume for the moment that $\Re(s) >>0$. Then $\det \xi^{s,\epsilon}$, extended by $0$ on $\{\xi\in V^*, {\det}^*\xi = 0\}$ is a smooth enough function on $V$ which vanishes on ${V^*}^\times$ up to any given (high) order. Similarly  for $\det \zeta^{t,\eta}$, so that the stronger statement of the Proposition is valid for $\Re(s), \Re(t)>>0$. Moreover, the operator $D_{s,t}$ has  polynomial coefficients which are also polynomial functions in the parameters $(s,t)$. The statement follows for generic $s, t$ by analytic continuation in $(s,t)$.
 \end{proof}
 
 The proposition can be translated in the terminology introduced for the symbolic calculus of the previous section.
 
 \begin{proposition}
 \begin{equation}\label{defdst}
 \det(x-y)\, \flat \, \det\xi ^{s,\epsilon} \det\zeta^{t,\eta} = \det\xi ^{s-1,\epsilon} \det\zeta^{t-1,\eta} d_{s,t}(\xi, \zeta, x,y)\ ,
 \end{equation}
 where $d(\xi, \zeta, x,y)$ is the symbol of the operator $D_{s,t}$.
 \end{proposition}

 \begin{proposition} Let $f_{\lambda, \mu}^*$ be the polynomial function on $(V^* \times V^*)\times(V\times V)$ defined by
 \[f_{\lambda, \mu} ^*(\xi, \eta, x,y) = f_{\lambda, \mu} (x,y, \xi, \eta)\ .\]
 Then the following identity holds
 \begin{equation}\label{fstar=d}
 f_{\lambda, \mu}^*(\xi, \eta, x,y) = \kappa_5\, (\lambda, \mu)\, d_{\lambda-\frac{n}{r}+1, \mu-\frac{n}{r}+1}(\xi, \eta, x,y)\ ,
 \end{equation}
 where $\kappa_5(\lambda, \mu)$ is a meromorphic $\not \equiv 0$ function on $\mathbb C\times \mathbb C$.
 \end{proposition}
 \begin{proof} Use Proposition \ref{compduality} to transform the identity \eqref{flambdamu1} and obtain
 \begin{equation}\label{keyeq}
 \begin{split}
 \kappa_3(\lambda,\epsilon; \mu, \eta) \det(x-y)\,  \flat\, \left( k_{\frac{2n}{r} -1-\lambda, -\epsilon}(\xi)\,  k_{\frac{2n}{r} -1-\mu, -\eta}(\zeta)\right)\\
 = \big(k_{\frac{2n}{r}-\lambda, -\epsilon} (\xi)\, k_{\frac{2n}{r}-\mu, -\eta}(\zeta)\big)\,  f^*_{\lambda, \mu}(\xi, \zeta, x,y)  \ .
 \end{split}
 \end{equation}
 Assume for a moment that $\Re(\lambda), \Re(\mu) >>0$. Consider the right hand side of \eqref{keyeq}. Recalling Proposition \ref{Fintw}, $k_{\frac{2n}{r}-\lambda, -\epsilon}$ is a linear combination of  of the two distributions $(\det \xi)^{-\frac{n}{r}+1+\lambda,\pm}$. As $\Re(\lambda)$ is supposed to be veery large,  this distribution is a function, smooth up to a large order. Now  on the open set $\{\det \xi>0\}\times \{ \det \zeta>0\}$, $(\det \xi)^{s,\pm} = \det \xi^s$, so that on this open set
  \[k_{\frac{2n}{r}-\lambda,-\epsilon}(\xi)= \kappa_6(\lambda-\epsilon) (\det \xi)^{-\frac{n}{r}+1+\lambda}\, ,\]
  where $\kappa_6$ is a meromorphic function of $\lambda$.
 
 A similar result holds for $k_{\frac{2n}{r}-\mu, -\epsilon}$, so that the right hand side of \eqref{keyeq} can be rewritten as
 \begin{equation}\label{RHS}
 \kappa_6(\lambda,-\epsilon) \kappa_6(\mu,-\eta)(\det \xi)^{-\frac{n}{r}+1+\lambda}(\det \zeta)^{-\frac{n}{r}+1+\mu}f^*_{\lambda, \mu}(\xi, \zeta, x,y)\ .
 \end{equation}
 Consider now the left hand side of \eqref{keyeq}. By the previous considerations, on the open set $\{\det \xi >0\}\times \{ \det \zeta>0\}$
 \[k_{\frac{2n}{r} -1-\lambda, -\epsilon}(\xi)k_{\frac{2n}{r} -1-\mu, -\eta}(\zeta)= \kappa_6(\lambda+1,\epsilon) \kappa_6(\mu+1,\eta)(\det \xi)^{-\frac{n}{r} +\lambda}(\det \zeta)^{-\frac{n}{r} +\mu},
 \] 
 so that, using  \eqref{defdst}, the left hand side of \eqref{keyeq} can be rewritten  as
 \begin{equation}\label{LHS}
 \kappa_6(\lambda+1,\epsilon) \kappa_6(\mu+1,\eta) (\det \xi)^{-\frac{n}{r}+1 +\lambda} (\det \zeta)^{-\frac{n}{r}+1 +\mu} d_{-\frac{n}{r} +\lambda, -\frac{n}{r} +\mu} (\xi,\zeta, x,y)\ .
 \end{equation} 
 Now compare \eqref{RHS} and \eqref{LHS} to conclude that both sides of \eqref{fstar=d} coincide on the open set  $\{\det \xi >0\}\times \{ \det \zeta>0\}$. As  both $f^*$ and $d$ are polynomial functions, the result follows everywhere on $V^*\times V^*$.

By analytic continuation, the conditions $\Re(\lambda) >>0$ and $\Re(\mu)>>0$ can be removed, thus finishing the proof of the proposition. 
 \end{proof}

 For the last part of this article, renormalize the operator $F_{\lambda,\mu}$ by demanding that its symbol $f_{\lambda, \mu}$ satisfies 
 \begin{equation}\label{normf}
 f_{\lambda, \mu}(x,y,\xi,\zeta) = d_{\lambda-\frac{n}{r}+1, \mu-\frac{n}{r}+1}(\xi, \zeta, x,y)\ .
 \end{equation}
 
 The next proposition gathers the main properties of the operator $F_{\lambda, \mu}$ which have been obtained.
 
 \begin{proposition} {\ }
 \smallskip
 
 $i)$ the operator $F_{\lambda, \mu}$ satisfies the covariance relation, valid for any $g\in G$
 \begin{equation}\label{covF}
F_{\lambda \mu} \circ \big(\pi_{\lambda, \epsilon} (g) \otimes \pi_{\mu,\eta}(g) \big) = \big(\pi_{\lambda+1,- \epsilon} (g) \otimes \pi_{\mu+1,-\eta}(g)\big)\circ 
F_{\lambda, \mu}
 \end{equation}
 \indent
$ii)$ the coefficients of $F_{\lambda,\mu}$ are polynomial functions on $V\times V$,  depending only on $(x-y)$.
\smallskip

$iii)$ the coefficients of $F_{\lambda, \mu}$ depend polynomially on the parameters $\lambda, \mu$.
 \end{proposition}
 
 The fact that the coefficients of $F_{\lambda, \mu}$ depend only on $(x-y)$ can be deduced from the covariance property \eqref{covF} when applied to translations by elements of $V$, acting  $V\times V$ by $(x,y) \longmapsto (x+v, y+v)$.
  
 The operator $F_{\lambda, \mu}$ is called the \emph{source operator} and following the general pattern, can be used for constructing covariant bi-differential operators.
  
\section{The  covariant bi-differential operators}

Let $\res : \mathcal S(V\times V) \longrightarrow \mathcal S(V)$ be the restriction operator from $V\times V$ to $\diag(V) = \{ (x,x), x\in V\}\simeq V$ given by
\[\res(f)(x) = f(x,x),\qquad \text{for } x\in V\ .
\]
\begin{proposition} For $(\lambda, \epsilon), (\mu,\eta)\in \mathbb C\times \{\pm\}$, for anyn $g\in G$
\[\res \circ (\pi_{\lambda, \epsilon}(g) \otimes \pi_{\mu, \eta} (g)= \pi_{\lambda+\mu, \epsilon \eta}(g) \circ \res
\]
\end{proposition}
The proof is elementary and left to the reader.

For any positive integer $k$, let
\[F^{(k)}_{\lambda, \mu} = F_{\lambda+k-1,\mu+k-1} \circ \dots \circ F_{\lambda,\mu}
\]
and
\[B^{(k)}_{\lambda,\mu} = \res \circ F^{(k)}_{\lambda,\mu}\ .
\]
\begin{proposition} The operators $B_{\lambda, \mu}^{(k)}$ satisfy the following covariance relation, valid for any $g\in G$
\begin{equation}
B^{(k)}_{\lambda,\mu} \circ \big(\pi_{\lambda, \epsilon}(g)\otimes \pi_{\mu, \eta}(g)\big)=
\pi_{\lambda+\mu+2k, \epsilon \eta}(g) \circ B^{(k)}_{\lambda,\mu}\ .
\end{equation}
\end{proposition}
As the coefficients of $F^{(k)}_{\lambda,\mu}$ depend only on $x-y$, the bi-differential operator $B^{(k)}_{\lambda,\mu}$ has constant coefficients. It is also a consequence of the covariance relation when applied to translations by elements of $V$.

There is a natural notion of symbol for a bi-differential operator $B : C^\infty(V\times V ) \longrightarrow C^\infty(V)$ (say with polynomial coefficients), extending the classical definition by letting
\[B \big(e^{i(x.\xi+y.\zeta)}\big)_{x=y} = b(x,\xi, \zeta) e^{i(x.\xi+x.\zeta)}\ .
\]
\begin{proposition}
The  symbol  of the operator $B^{(k)}_{
\lambda,\mu}$ denoted by $b_{\lambda,\mu}^{(k)}(\xi, \zeta)$ is equal to
\begin{equation}
b_{\lambda,\mu}^{(k)}(\xi, \eta) = f_{\lambda,\mu}^{(k)}(0,0, \xi, \eta)\ .
\end{equation}
\end{proposition}
\begin{proof} This is a consequence of the fact that the coefficents of $F_{\lambda, \mu}$ and hence of $F^{(k)}_{\lambda, \mu}$ only depend on $(x-y)$. When restricting to the diagonal $\{x=y\}$, all terms of $F_{\lambda, \mu}$ cancel except those with constant coefficients.
\end{proof}

In the next statements and proofs,  the signs $\pm$ will be omitted. Identities are proved on $\{\det \xi>0\} \times \{ \det \zeta>0\}$ and then extended to $V^*\times V^*$.

\begin{proposition}
The symbols $b_{\lambda,\mu}^{(k)}$ satisfy the following recurrence relation
\begin{equation}\label{recb}
\det \xi^{\lambda-\frac{n}{r}} \det \zeta^{\mu-\frac{n}{r}} \,b_{\lambda,\mu}^{(k)}(\xi, \zeta)=  \det \left(\frac{\partial }{\partial \xi} - \frac{\partial }{\partial \zeta} \right)\left(\det \xi^{\lambda-\frac{n}{r}+1} \det \zeta^{\mu-\frac{n}{r}+1} b^{(k-1)}_{\lambda+1,\mu+1}(\xi, \zeta \right) \ .
\end{equation}
\end{proposition}
\begin{proof}
\[B_{\lambda,\mu}^{(k)} = \res\circ \left(F^{(\lambda+k-1,\mu+k-1)}\circ \dots \circ F_{\lambda+1,\mu+1}\right) \circ F_{\lambda,\mu}
\]
\[= \res \circ F^{k-1}_{\lambda+1,\mu+1} \circ F_{\lambda,\mu}
\]
and its symbol satisfies
\[b^{(k)}_{\lambda,\mu} (\xi, \zeta) = \left(b_{\lambda+1,\mu+1}^{(k-1)}\# f_{\lambda,\mu} \right)(0,0,\xi, \zeta)\ .
\]
Now use Proposition \ref{compduality} and the definition \eqref{normf} of the symbol of $F_{\lambda, \mu}$ to rewrite this last equation as
\[b_{\lambda,\mu}(\xi,\zeta) =\left(d_{\lambda-\frac{n}{r}+1,\mu-\frac{n}{r}+1} \,\flat \, b_{\lambda+1, \mu+1}^{(k)}\right)(\xi, \zeta, 0,0)
\]
and by Proposition \ref{x=0} 
\[b_{\lambda,\mu} (\xi, \zeta)= D_{\lambda-\frac{n}{r}+1,\mu-\frac{n}{r}+1}\, (b_{\lambda+1,\mu+1}^{k-1}) (\xi, \zeta)
\]
Hence
\[\det \xi^{\lambda-\frac{n}{r}} \det \zeta^{\mu-\frac{n}{r}} b_{\lambda,\mu} (\xi, \zeta) = (\det \xi^{\lambda-\frac{n}{r}} \det \zeta^{\mu-\frac{n}{r}} D_{\lambda-\frac{n}{r}+1,\mu-\frac{n}{r}+1}) b_{\lambda+1,\mu+1}^{(k-1)} (\xi, \eta)
\]
\[= \left(\det\left(\frac{\partial}{\partial \xi} - \frac{\partial}{\partial \zeta}\right)\circ \det \xi ^{\lambda-\frac{n}{r}+1} \det \zeta^{\mu-\frac{n}{r}+1} \right) b_{\lambda+1,\mu+1}^{(k-1)} (\xi, \zeta)
\]
and by \eqref{defDst}, and the proposition follows.
\end{proof}

\section{The Rodrigues formula}

The \emph{Rodrigues formula} is a type of formula which is valid for many orthogonal polynomials or special functions of one variable (see \cite{gr}, formulas 8.960.1, 8.939.7, 8.949. 7 and 8, 8.959.1). These formulas imply recurrence relations which are of the same type as the recurrence relation satisfied by the symbols $d_{\lambda, \mu}^{(k)}$. This remark allows to solve the recurrence relation and to determine complely the symbols.

\begin{theorem} Let $s,t\in \mathbb C$. For any integer $k$, there exists a unique polynomial $c^{(k)}_{s,t}$ such that
\begin{equation}\label{Rodrigues}
\det\left(\frac{\partial}{\partial \xi}- \frac{\partial}{\partial \zeta}\right)^k( \det \xi^{s+k}\det \zeta ^{t+k})
= c^{(k)}_{s,t}(\xi,\zeta) \det \xi^s \det \zeta^t\ .
\end{equation}
\end{theorem}
\begin{proof} This is a consequence of Theorem 4.8 in \cite{bck}.
\end{proof}
\begin{proposition}
The polynomials $c_{s,t}^{(k)}$ satisfy the recurrence relation
\begin{equation}\label{recc}
\det \xi^s \det \zeta^t c_{s,t}^{(k)} (\xi,\zeta) =\det\left(\frac{\partial}{\partial \xi}- \frac{\partial}{\partial \zeta}\right)\big(\det \xi^{s+1} \det \zeta^{t+1} c^{(k-1)}_{s+1,t+1} (\xi,\zeta)\big)
\end{equation}
\end{proposition}
\begin{proof}
\[\det\left(\frac{\partial}{\partial \xi}- \frac{\partial}{\partial \zeta}\right)^k( \det \xi^{s+k}\det \zeta ^{t+k})\]
\[=\det\left(\frac{\partial}{\partial \xi}- \frac{\partial}{\partial \zeta}\right)\left(\det\left(\frac{\partial}{\partial \xi}- \frac{\partial}{\partial \zeta}\right)^{k-1} \det \xi^{s+1+k-1} \det \zeta^{t+1+k-1}
\right)
\]
\[= \det\left(\frac{\partial}{\partial \xi}- \frac{\partial}{\partial \zeta}\right)\big(\det \xi^{s+1} \det \zeta^{t+1} c^{(k-1)}_{s+1,t+1} (\xi,\zeta)\big)\ .
\]
\end{proof}
\begin{proposition} For $\lambda, \mu\in \mathbb C$ and $k\in N$,
\begin{equation}
d_{\lambda, \mu}(\xi, \zeta) = c_{\lambda-\frac{n}{r}, \mu-\frac{n}{r})}^{(k)}(\xi, \zeta)
\end{equation}
\end{proposition}
\begin{proof} The two families of polynomials $\left(d^{(k)}_{\lambda, \mu}\right)_{k\in \mathbb N}$ and $\left(c_{\lambda-\frac{n}{r}, \mu-\frac{n}{r}}^{(k)}\right)_{k\in \mathbb N}$ satisfy the same recurrence relation. Moreover,
$ c^{(0)}_{\lambda, \mu}(\xi, \zeta) = d_{\lambda, \mu}^{(0)}(\xi, \zeta) = 1$, so that by induction on $k$ the two families coincide. 
\end{proof}

\begin{theorem} Let $c^{(k)}_{s,t}$ be the polynomial defined by the Rodrigues formula 
\eqref{Rodrigues}. Let $B^{(k)}_{\lambda, \mu}$ be the bi-differential operator on $V\times V$ defined by
\[B^{(k)}_{\lambda, \mu} = \res\,\circ\, c^{(k)}_{\lambda-\frac{n}{r}, \mu-\frac{n}{r}}\left(\frac{\partial}{\partial x}, \frac{\partial}{\partial y}\right)\ .
\]
Then, for any $\epsilon, \eta\in \{ \pm \}$
\begin{equation}
B_{\lambda,\mu}^{(k)}\, \circ \,\Big(\pi_{\lambda, \epsilon}(g)\otimes \pi_{\mu, \eta}(g)\Big) = \pi_{\lambda+\mu+2k, \epsilon \eta}(g) \circ B_{\lambda,\mu}^{(k)}\ .
\end{equation}\ .
\end{theorem}

\section{Some examples}

There are some situations were the source operator is explicitly known, so that the Rodrigues formula can  be obtained in a direct calculation.

For the classical Rankin-Cohen brackets, which in the present approach corresponds to the case $V=\mathbb R$, the source operator is obtained in \cite{c2} Section 5 and deduced from the \emph{Cayley operator}. The symbols of the Rankin-Cohen operators are known to be related to  the Jacobi polynomials (see \cite{kp16}). The Rodrigues formula obtained in this article is shown to correspond to the classical Rodrigues formula for the Jacobi polynomials. 

The source operator is also known for the conformally covariant bi-differential operators on $\mathbb R^n$. The Jordan algebra is $V=\mathbb R^n$ with the Jordan multiplication
\[(x_1,x_2,\dots, x_n)(y_1,y_2,\dots, y_n)\]\[ = \big((x_1y_1-x_2y_2-\dots -x_ny_n), x_1y_2+x_2y_1,\dots, x_1y_n+x_ny_1\big)
\] 
The conformal group is $SO_0(1,n+1)$ and 
the determinant  is the quadratic form
\[\det x = x_1^2+x_2^2+\dots+x_n^2\ .
\]
The source operator was computed in \cite{bc}, see also \cite{bck} Section 10. The Rodrigues formula is new, and in this case there is no known relation to a Rodrigues formula for a family of orthogonal polynomials is known.

Our third example is concerned with the Juhl operators.  The geometric context is different, but still in the context of symmetry breaking differential operators. The Juhl operators are differential operators  from $S^n$ to $S^{n-1}$  and they are covariant with respect to the conformal group of $S^{n-1}$ viewed as the subgroup of the conformal group of $S^n$ which stabilizes $S^{n-1}\subset S^n$. The symbols of these operators were already known to be connectd with the Gegenbauer polynomials (see \cite{j, kp16}).

\subsection{Rodrigues formula for the symbols of the classical Rankin-Cohen brackets}

Let $V=\mathbb R$ with its usual product. The group $G$ is equal to $SL(2,\mathbb R)$, and the determinant is given by  $\det x = x$. The representations $\pi_{\lambda, \epsilon}$ are given by
\[\pi_{\lambda, \epsilon}(g) f(x) = (cx+d)^{-\lambda, \epsilon} f\left((ax+b)(cx+d)^{-1}
\right)\ .\]

The source operator $F_{\lambda,\mu}$ is given by

\begin{equation}
F_{\lambda, \mu} = (x-y) \frac{\partial^2}{\partial x \partial y}-\mu \frac{\partial}{\partial x} +\lambda \frac{\partial}{\partial y} 
\end{equation}
and its symbol is given by
\begin{equation}\label{symblambdamu}
f_{\lambda, \mu} (x,y,\xi,\eta) =  -(x-y)\,\xi\eta+i(-\mu\xi+\lambda \eta)
\end{equation}
The symbols $b_{\lambda, \mu}^{(k)}$ of the Rankin-Cohen brackets satisfy

\[b^{(k)}_{\lambda, \mu}(\xi, \zeta) =\big( b^{(k-1)}_{\lambda+1,\mu+1}(\xi, \zeta)\, \#\,f_{\lambda, \mu}(x,y,\xi,\zeta)\big)(0,0,\xi, \zeta)
\]
Now use \eqref{symblambdamu} and the composition formula \eqref{compKp} to get

\begin{equation}\label{RC1}
b_{\lambda, \mu}^{(k)} =i \Bigg((-\mu \xi +\lambda \zeta)\, b_{\lambda+1, \mu+1}^{(k-1)}+\xi\zeta\left( \frac{\partial}{\partial \xi}-\frac{\partial}{\partial \zeta}\right)\Bigg) b_{\lambda+1, \mu+1}^{(k-1)} 
\end{equation}

The next proposition  introduces a family of polynomials which will be shown to solve the recursion relation.
\begin{proposition} Let $\alpha, \beta\in \mathbb C$. For any $l\in \mathbb N$, there exists a (unique) polynomial $q_l^{\alpha,\beta}(\xi,\eta)$, homogeneous of degree $l$ such that 
\begin{equation}
\left(  \frac{\partial}{\partial \xi}-  \frac{\partial}{\partial \eta}\right)^l \xi^{\alpha+l} \eta^{\beta+l} =\xi^\alpha \eta^\beta\, q_l^{\alpha, \beta} (\xi,\eta)\ .
\end{equation}

\end{proposition}

\begin{proposition} The polynomials $q_k^{\alpha, \beta}$ satisfy the following  relation
\begin{equation}\label{RC2}
q_l^{\alpha,\beta} = \xi\eta\left( \frac{\partial}{\partial \xi}-  \frac{\partial}{\partial \eta}\right) q_{l-1}^{\alpha+1, \beta+1} +\big(-(\beta+1) \xi+(\alpha+1) \eta\big)\,  q^{\alpha+1,\beta+1}_{l-1}
\end{equation}
\end{proposition}
\begin{proof} Observe that
\[\left( \frac{\partial}{\partial \xi}-  \frac{\partial}{\partial \eta}\right)^l (\xi^{\alpha+l} \eta^{\beta+l})= 
\left( \frac{\partial}{\partial \xi}-  \frac{\partial}{\partial \eta}\right) \left( \left( \frac{\partial}{\partial \xi}-  \frac{\partial}{\partial \eta}\right)^{l-1} \xi^{\alpha+1+l-1}\eta^{\beta+1+l-1} \right)\]
\[ = \left( \frac{\partial}{\partial \xi}-  \frac{\partial}{\partial \eta}\right) \left(\xi^{\alpha+1}\eta^{\beta+1}\, q_{l-1}^{\alpha+1,\beta+1}\right)(\xi,\eta) 
\]
and use Leibniz rule to conclude.
\end{proof}
\begin{theorem}\label{theoremRC1}
 For $\lambda, \mu\in \mathbb C$,  and $k\in \mathbb N$
\begin{equation}\label{RC1}
b_{\lambda, \mu}^{(k)}(\xi, \eta) = q_k^{\lambda-1,\mu-1}(i\xi, i\zeta) = i^kq_k^{\lambda-1,\mu-1} (\xi, \zeta)\ .
\end{equation}
\end{theorem}
\begin{proof}
Set $\alpha = \lambda-1, \beta = \mu-1$ and compare \eqref{RC1} and \eqref{RC2} to show that $b_{\lambda,\mu}^{(k)}(\xi, \zeta)$ and $q^{\alpha, \beta}_k(i\xi, i\zeta)$ satisfy the same recurrence relation. As $b^{(0)}_{\lambda, \mu} = q_0^{\lambda-1, \mu-1} = 1$,  the conclusion follows  by induction on $k$.
\end{proof}

\medskip

Finally, the polynomials $q^{\alpha, \beta}_k$ are closely related to the \emph{Jacobi polynomials}. First recall the Rodrigues formula  which can be taken as a definition of the Jacobi polynomials (see \cite{gr} 8.960.1 p. 998).  
\begin{equation}
(1-t)^\alpha (1+t)^\beta P_k^{\alpha, \beta} (t) = \frac{(-1)^k}{ 2^kl!} \left( \frac{d}{dt}\right)^k\big((1-t)^{k+\alpha}(1+t)^{k+\beta} \big)
\end{equation}

In \cite{kp16}, the authors define a family of homogeneous polynomials $\widetilde P_k^{\alpha, \beta}$ of two variables by the formula
\begin{equation}\label{KPJacobi}
\widetilde P_k^{\alpha, \beta}(\xi,\eta) = (-1)^k (\xi+\eta)^k P_k^{\alpha, \beta} \left(\frac{\eta-\xi}{\xi+\eta} \right)
\end{equation}
\begin{proposition} Let $\alpha, \beta\in \mathbb C$. Then for all $l\in \mathbb N$
\begin{equation}
Q_k^{\alpha, \beta} = (-1)^k\, k!\, \widetilde P_k^{\alpha, \beta}
\end{equation}
\end{proposition}
\begin{proof}
For $F$ a function of two variables $(\xi, \eta)$, associate the function $f$ of one variable given by
\[f(t) = F(1-t,1+t)\ .
\]
Then,
\[\left(\frac{d}{dt}\right)^k f (t) =(-1)^k\,\left(\left(\frac{\partial}{\partial \xi} -  \frac{\partial}{\partial \eta}\right)^k F\right) (1-t,1+t) \ 
\]

Apply this result to the function $F(\xi,\eta) = \xi^{k+\alpha}\eta^{k+\beta}$, which corresponds to $f(t) = (1-t)^{k+\alpha}(1+t)^{k+\beta}$. On one hand, by Rodrigues formula
\[\left(\frac{d}{dt}\right)^k(1-t)^\alpha(1-t)^\beta = (-1)^k \,2^k\,k!\,(1-t)^\alpha(1+t)^\beta P_k^{\alpha, \beta}(t)
\]
whereas on the other hand,
\[\left(\left(\frac{\partial}{\partial \xi} -  \frac{\partial}{\partial \eta}\right)^k F\right)(1-t,1+t) = (1-t)^\alpha(1+t)^\beta Q_k^{\alpha,\beta}(1-t,1+t)\ .
\]

 Hence
\[2^k k! P_l^{\alpha, \beta}(t) =  Q_k^{\alpha, \beta} (1-t, 1+t)
\]
Now let $(\xi,\eta)\in \mathbb R^2$ such that $\xi+\eta=2$. Write $\xi=1-t$ and $\eta=1+t$, so that 
$\displaystyle t = \frac{\eta-\xi}{\xi+\eta}$. 
\[\, k! \,(\xi+\eta)^k\, P_k^{\alpha,\beta}\left(\frac{\eta-\xi}{\xi+\eta}\right)= Q_k^{\alpha, \beta} (\xi,\eta)\ ,
\]
or equivalently, using \eqref{KPJacobi}
\[(-1)^k\,k!\, \widetilde P_k^{\alpha, \beta}(\xi, \eta) = Q_k^{\alpha, \beta}(\xi, \eta)
\] 
whenever $\xi+\eta=2$. As both sides are homogeneous polynomials of degree $k$, the conclusion follows.
\end{proof}
\subsection{Conformally covariant bi-differential operators}

Let  $V=\mathbb R^n$ with the Jordan multiplication
\[(x_1,x_2,\dots, x_n)(y_1,y_2,\dots, y_n)\]\[ = \big((x_1y_1-x_2y_2-\dots -x_ny_n), x_1y_2+x_2y_1,\dots, x_1y_n+x_ny_1\big)
\] 
The group is $SO_0(1,n+1)$ and 
the determinant  is the quadratic form
\[\det x = x_1^2+x_2^2+\dots+x_n^2\ .
\]
In this case, there is no use to introduce the signs $\pm$. Notice in particular that the determinant is everywhere positive. Hence the representations to be considered belong to the scalar principal series and are given by
\[\pi_\lambda(g) f(x) = \kappa(g^{-1},x)^\lambda f(g^{-1}(x))\ ,
\]
where $\kappa(g,x)$ is the conformal factor of $g$ at $x$.

The source operator was computed in \cite{bc}, see also \cite{bck} Section 10.

 Denote by $\Delta_x$ (resp. $\Delta_y$) the Laplacian on $\mathbb R^n$ with respect to the variable $x$ (resp. $y$) and let
 \[\nabla_{xy} = \sum_{j=1}^n \frac{\partial ^2}{\partial x_j \partial y_j} \ .\]
 
 The source operator in this case is the differential operator $F_{\lambda, \mu}$ on $\mathbb R^n\times \mathbb R^n$ given by
 \begin{equation}
 \begin{split}
 F_{\lambda,\mu} &= \vert x-y\vert^2 \Delta_x \Delta_y\\
 +2(2\lambda-n+2))&\sum_{j=1}^n (x_j-y_j) \frac{\partial}{\partial x_j}\Delta_y+2(2 \mu-n+2) \sum_{j=1}^n (y_j-x_j) \frac{\partial}{\partial y_j}\Delta_x\\
 +2\mu(2\mu-n+2)& \Delta_x-2(2\lambda-n+2)(2\mu-n+2) \nabla_{x,y} + 2\lambda(2\lambda-n+2)\Delta_y\ .
 \end{split}
 \end{equation} 
 
The symbol $f_{\lambda, \mu}$ of the source operator $E_{\lambda, \mu}$   is given by
\begin{equation}
\begin{split}
f_{\lambda, \mu}(x,y,\xi, \zeta) &= \vert x-y\vert^2 \vert \xi\vert^2 \vert \zeta \vert^2\\
-i\Big(2(2\lambda -n+2) &\sum_{j=1}^n (x_j-y_j)\xi_j\vert \zeta\vert^2 + 2( 2\mu-n+2) \sum_{j=1}^n (y_j-x_j)\zeta_j \vert \xi\vert^2\Big)\\
-\Big(+2\mu(2\mu-n+2)\vert \xi\vert^2& -2(2\lambda-n+2)(2\mu-n+2) \xi.\zeta+2\lambda(2\lambda-n+2)\vert\zeta\vert^2\Big)\ .
\end{split}
\end{equation}
As before
 \[b_{\lambda, \mu}^{(k)}(\xi, \zeta) = B_{\lambda+1,\mu+1}^{(k-1)}\sharp f_{\lambda, \mu}) (0,0,\xi, \zeta)\]
which after computation amounts to the following recurrence relation 
\begin{equation}\label{recbi}
\begin{split}
 &b^{(k)}_{\lambda, \mu} (\xi,\zeta) =  i\\
 \Bigg(\Big(2\mu(2\mu-n+2)\vert \xi\vert^2 &-2(2\lambda-n+2)(2\mu-n+2) \xi.\zeta+2\lambda(2\lambda-n+2)\vert\zeta\vert^2\Big)b^{(k-1)}_{\lambda+1, \mu+1}(\xi,\zeta)\\
 &+2(2\lambda-n+2)\vert \zeta\,\vert^2\sum_{j=1}^n \xi_j \left( \frac{\partial}{\partial \xi_j} - \frac{\partial}{\partial \zeta_j}\right)  b^{(k-1)}_{\lambda+1, \mu+1}(\xi,\zeta) \\
 &+2(2\mu-n+2)\vert \xi\vert^2\sum_{j=1}^n \eta_j\left( \frac{\partial}{\partial \zeta_j} - \frac{\partial}{\partial \xi_j}\right)  b^{(k-1)}_{\lambda+1, \mu+1}(\xi,\zeta)
 \\
 &+\vert \xi\vert^2 \vert \zeta\vert^2\,q\left(\frac{\partial}{\partial \xi} - \frac{\partial}{\partial \zeta}\right)b^{(k-1)}_{\lambda+1,\mu+1}(\xi,\zeta) \Bigg)\ .
 \end{split} 
\end{equation}
 
 Together with the condition $d_0^{\lambda, \mu} \equiv 1$, the recurrence relation \eqref{recbi} determines $d_k^{\lambda, \mu}$ by induction on $k$. To solve this recurrence relation, let us introduce for  $\alpha, \beta\in \mathbb C$ the family of polynomials $p_k^{\alpha, \beta}$  on $(\mathbb R^n\times \mathbb R^n)^*$ defined by
 \begin{equation}
q\left( \frac{\partial}{\partial \xi} -  \frac{\partial}{\partial \eta}\right)^k\left( \vert \xi\vert^{2(\alpha +k)} \vert \eta\vert^{2(\beta+k}\right)  = p_k^{\alpha,\beta}(\xi,\eta)\, \vert \xi\vert^{2\alpha} \vert \eta\vert^{2\beta}\ .
 \end{equation}
 It is easily seen that $p_k^{\alpha, \beta}$ thus defined is a homogeneous polynomial of degree $2k$.
 
 \begin{proposition} The polynomials $p_k^{\alpha, \beta}$ satisfy the recurrence relation
 \begin{equation}\label{recbi2}
 \begin{split}
 &p_k^{\alpha, \beta}(\xi, \eta) = \\
 \vert \xi\vert^2 \vert \eta\vert^2&q\left( \frac{\partial}{\partial \xi} -  \frac{\partial}{\partial \eta}\right)p_{k-1}^{\alpha+1,\beta+1}\\
 +2(2\alpha+2)\vert\eta\vert^2& \sum_{j=1}^n \xi_j \left(\frac{\partial}{\partial \xi_j}- \frac{\partial}{\partial \eta_j}\right) p_{k-1}^{\alpha+1,\beta+1}\\
 +2(2\beta+2)\vert\xi\vert^2 &\sum_{j=1}^n \eta_j \left(\frac{\partial}{\partial \eta_j}- \frac{\partial}{\partial \xi_j}\right) p_{k-1}^{\alpha+1,\beta+1}\\
 +\Big((2\alpha+2) (2\alpha+n) &\vert \eta\vert^2-2(2\alpha+2)(2\beta+2)\langle \xi,\eta\rangle+(2\beta+2)(2\beta +n) \vert \xi\vert^2\Big) p_{k-1}^{\lambda+1,\mu+1}\\
 \end{split}
 \end{equation}
 \end{proposition}
 \begin{proof} Observe that
 \[q\left( \frac{\partial}{\partial \xi} -  \frac{\partial}{\partial \eta}\right)^k\vert \xi\vert^{2(\alpha +k)} \vert \eta\vert^{2(\beta+k)}\]\[= q\left( \frac{\partial}{\partial \xi} -  \frac{\partial}{\partial \eta}\right) \left(q\left( \frac{\partial}{\partial \xi} -  \frac{\partial}{\partial \eta}\right)^{k-1} \vert \xi\vert^{2(\alpha+1+(k-1)}\vert \eta\vert^{2(\beta+1+(k-1)} \right)
 \]
 \[=q\left( \frac{\partial}{\partial \xi} -  \frac{\partial}{\partial \eta}\right)\left(p_{k-1}^{\alpha+1, \beta+1} \vert \xi\vert^{2(\alpha+1)} \vert \eta\vert^{2(\beta+1)}\right) \ .
 \]
 A straightforward calculation done \emph{asinus trottans} yields the result.
 \end{proof}
 
 \begin{theorem}\label{theoremBC1}
 For $\lambda, \mu\in \mathbb C$ and for $k\in \mathbb N$,
 
 \begin{equation} 
 b^{(k)}_{\lambda,\mu} (\xi,\zeta)= p_k^{\lambda-\frac{n}{2},\mu-\frac{n}{2}}(i\xi, i\zeta)= i^k p(\xi, \zeta) .
 \end{equation}
 \end{theorem}
 \begin{proof} As $d_{\lambda, \mu}^{(0)} = 1$ and $p_0^{\alpha, \beta} = 1$, the conclusion follows by comparing \eqref{recbi} and \eqref{recbi2} after setting $\alpha = \lambda-\frac{n}{2}, \beta = \mu-\frac{n}{2}$.
 
 \end{proof}

\subsection{The Juhl operators }

Some years ago, A. Juhl (see \cite{j}) introduced a family of differential operators from $\mathbb R^n$ to $\mathbb R^{n-1}$ which are covariant for the subgroup of the conformal group of $\mathbb R^n$ which preserves the hyperplane $\mathbb R^{n-1}$. Recently, I presented a new approach to these operators (see \cite{c}), based on the source operator method. 

The group $G=SO_0(1,n+1)$ acts conformally on $\mathbb R^n$ by a rational action. For $g\in G$ defined at $x\in \mathbb R^n$, let $\kappa(g,x)$ be the conformal factor of $g$ at $x$, so that for every $v\in \mathbb R^n$
\[\forall v\in \mathbb R^n, \qquad \vert Dg(x) v \vert = \kappa(g,x) \vert v\vert\ .
\]
For $\lambda\in \mathbb C$, the \emph{principal series representation} $\pi_\lambda$ of $G$ (in the noncompact picture) is given by
\[\pi_\lambda(g) f (x) = \kappa(g^{-1},x)^\lambda f\big(g^{-1}(x)\big)
\]
where $f\in C^\infty(\mathbb R^n)$. 

Let identify the hyperplane $\{ {\bf x} \in \mathbb R^n, x_n=0\}$ with $\mathbb R^{n-1}$ and write $x=(x',x_n)$ where $x'\in \mathbb R^{n-1}$. The subgroup $H$ of $G$ which stabilizes this hyperplane can be identified with $SO_0(1,n)$. For $\mu\in \mathbb C$, the scalar principal series  representation $\pi'_\mu$ of $H$ is realized on $C^\infty(\mathbb R^{n-1})$ and given  by
\[\pi'_{\mu} (h) f(x') = \kappa(h^{-1},x')^\mu f\big(h^{-1}(x')\big), \qquad h\in H
\]
where $f\in C^\infty(\mathbb R^{n-1})$.

For $\lambda\in \mathbb C$, let $E_\lambda$ be the differential operator on $\mathbb R^n$ given by
\begin{equation}
E_\lambda = x_n \Delta+(2\lambda-n+2)\frac{\partial}{\partial x_n}\ . 
\end{equation}
where $\displaystyle \Delta = \sum_{j=1}^n \frac{\partial^2}{\partial x_j^2}$ is the usual Laplacian on $\mathbb R^n$. The operator $E_\lambda$ has polynomial coefficients and  its 
 is given by
\begin{equation}\label{symbE}
e_\lambda (x,\xi) = -x_n\vert \xi\vert^2+ i(2\lambda-n+2)\xi_n\ .
\end{equation}
\begin{proposition}
For any $h\in H$
\begin{equation}
E_\lambda \circ \pi_\lambda(h) = \pi_{\lambda+1}(h)\circ E_\lambda\ .
\end{equation}
\end{proposition}
See \cite{c2} for the proof. The operator $E_\lambda$ is called \emph{the source operator} and plays the same r\^ole for the construction of Juhl operators as  the source operator $F_{\lambda, \mu}$ did for the construction of the Rankin-Cohen operators. For $k\geq 1$ define
\begin{equation}
E_\lambda^{(k)} = E_{\lambda+k-1} \circ \dots \circ E_\lambda\ .
\end{equation}
The operator $E_\lambda^{(k)}$  has polynomials coefficients and its symbol is denoted by $e_\lambda^{(k)}$. It satisfies the covariance relation  
\begin{equation}
\forall h\in H, \qquad E_\lambda^{(k)} \circ \pi_\lambda(h) = \pi_{\lambda+k}(h)\circ E_\lambda^{(k)} \ .
\end{equation}

Notice that the coefficients of $E^{(k)}_\lambda$ depend only on the variable $x_n$, a consequence of the covariance relation when applied to the translations by vectors belonging to the hyperplane $\{ x_n=0\}$.

Define the restriction map $\res$ by
\[ C^\infty(\mathbb R^n) \ni f \longmapsto \res(f) \in C^\infty(\mathbb R^{n-1}), \qquad\res(f)(x') = f(x',0)\ .
\]
Notice that for any $\lambda\in \mathbb C$
 \[\forall h\in H,\qquad \res \circ \pi_\lambda(h) = \pi'_\lambda(h)\circ \res\]
For $k\in \mathbb N$, define
\begin{equation}\label{Jop} 
J_\lambda^{(k)} =\res\circ E_\lambda^{(k)} = \res \circ E_{\lambda+k-1} \circ \dots\circ E_\lambda\ .
\end{equation}
\begin{proposition} For any $h\in H$,
\begin{equation}\label{covJ}
J_\lambda^{(k)} \circ \pi_\lambda(h) = \pi'_{\lambda+k}(h) \circ J_\lambda^{(k)}\ .
\end{equation}
\end{proposition}

 The operator $J_\lambda^{(k)}$ is a differential operator from $\mathbb R^n$ into $\mathbb R^{n-1}$ with polynomial coefficients. Moreover,  $J_\lambda^{(k)}$ has constant coefficients, as a consequence of the covariance property \eqref{covJ} for $h$ a translation along a vector in $\mathbb R^{n-1}$. Hence the symbol $j_\lambda^{(k)}$ of $J_\lambda^{(k)}$ depends only on $\xi\in \mathbb R^n$.

The  definition of the operators $J^{(k)}_\lambda$ implies  a recurrence relation for their symbols. 
\begin{proposition} The polynomials $j^{(k)}_\lambda$ satisfy the following relation
\begin{equation}\label{recJ}
j^{(k)}_\lambda = \frac{1}{i}(2\lambda-n+2)\, \xi_n \,j^{(k-1)}_{\lambda+1}+ \vert \xi\vert^2\frac{\partial}{\partial \xi_n} \, j^{(k-1)}_{\lambda+1} \ .
\end{equation} 

\end{proposition}
\begin{proof}
As already noted, the symbol of $E_\lambda$ depend only on $x_n$. The same is  true (and for the same reason) for the operators $E_\lambda^{(k)}$. As $J^{(k)}_\lambda = \res\circ E_\lambda^{(k)}$,
\begin{equation}\label{symbED}
 jk^{(k)}_\lambda(\xi)= e_{\lambda}^{(k)}(0,\xi)\ .
\end{equation}
Next
\begin{equation}\label{J1}
 E^{(k)}_\lambda = \left( E_{\lambda+k-1} \circ \dots \circ E_{\lambda+1}\right) \circ E_\lambda  = \circ \,E_{\lambda+1}^{(k-1)}\circ E_\lambda\\
\end{equation}
so that
\[e_\lambda^{(k)}  = e_{\lambda+1}^{(k-1)}\,\#\,e_\lambda 
\]
and hence
\[j^{(k)}_\lambda (\xi)= \big(j^{(k-1)}_{\lambda+1}\, \#\,e_\lambda\big)(0,\xi) \]
Use the composition formula \eqref{compKp} and \eqref{symbE} to get \eqref{recJ}.

\end{proof}
Together with the initial condition $j^{(0)}_\lambda \equiv 1$, the recurrence relation  \eqref{recJ} determines the polynomials $j^{(k)}_\lambda$ by induction over $k$. To solve this recurrence relation, define for $\gamma\in \mathbb C$  the sequence of polynomials $B_k^{\gamma}$ on $\mathbb R^n$ by 
\begin{equation}\label{defBgamma}
B_0^\gamma= 1,\qquad \left(\frac{\partial}{\partial \xi_n}\right)^k \vert \xi\vert^{2(\gamma+k)} = B_k^{\gamma}(\xi) \vert\xi\vert^{2\gamma}
\end{equation}

\begin{lemma} Let $\gamma\in \mathbb C$. For any $k\geq 1$,
\begin{equation}\label{recB}
B_k^\gamma=2(\gamma+1)\xi_nB_{k-1}^{\gamma+1}+ \vert \xi\vert^2\frac{\partial}{\partial \xi_n} B_{k-1}^{\gamma+1}\ .
\end{equation}
\end{lemma}
\begin{proof}

\[\frac{\partial}{\partial \xi_n}^k (\vert\xi\vert)^{2(\gamma+k)}= \frac{\partial}{\partial \xi_n} \left( \left(\frac{\partial}{\partial \xi_n}\right)^{k-1} \vert \xi \vert^{2(\gamma+1+(k-1)}\right)=\frac{\partial}{\partial \xi_n}\left( \vert \xi\vert^{2(\gamma+1)}B_{k-1}^{\gamma+1}(t)\right)
\]
\[=  \vert \xi\vert^{2(\gamma+1)}\,\frac{\partial}{\partial \xi_n}  B_{k-1}^{\gamma+1}\xi)+2(\gamma+1)\,\xi_n\,\vert \xi \vert^{2\gamma}\, B_{k-1}^{\gamma+1}(\xi)\ ,
\]
and the conclusion follows.
\end{proof}
\begin{theorem}\label{theoremJ1}
 Let $\lambda\in \mathbb C$ and $k\in \mathbb N$. The following identity issatisfied
\begin{equation}
p_k^\lambda (\xi)= B_k^{\lambda -\frac{n}{2}}(i\xi)\ .
\end{equation}
\end{theorem}
\begin{proof}
Notice that $p_0^\lambda = B_0^{\lambda-\frac{n}{2}}=1$, use the homogeneity of the polynomials and compare \eqref{recJ} and \eqref{recB}  for $\gamma = \lambda-\frac{n}{2}$.
\end{proof}

 Notice that Theorem {\bf A2} in the introduction is merely a reformulation of Theorem \ref{theoremJ1}.
 
The polynomials $B_k^\gamma$ are connected with the classical \emph{Gegenbauer polynomials}. In fact, the latter may be defined through the Rodrigues formula (see \cite{gr} page 993)

\begin{equation}\label{Gegen}
C_k^\lambda(t) = c_k(\lambda)(1-t^2)^{-(\lambda-\frac{1}{2})} \left(\frac{d}{dt}\right)^k(1-t^2)^{k+\lambda-\frac{1}{2}}
\end{equation}
where 
\[c_k(\lambda) = \frac{(-1)^k \Gamma(\lambda+\frac{1}{2})\Gamma(k+2\lambda)}{2^k k! \Gamma(2\lambda)\Gamma(k+\lambda+\frac{1}{2})}\ .
\]
Observe first that $B_k^\gamma$ is a homogeneous polynomial of degree $k$. Next, as $\vert\xi\vert^2= \vert \xi'\vert +\xi_n^2 $,  $B^\lambda_k(\xi',\xi_n)$ can be written as a polynomial in $\xi_n$ and $\vert \xi'\vert ^2$, so set
\[B^\gamma_k(\xi',\xi_n) = A_k^\gamma( \vert \xi'\vert, \xi_n)
\]
where $A_k^\gamma$ is a polynomial of two variables, homogeneous of degree $k$ and even in the first variable. With this notation, \eqref{defBgamma} implies
\begin{equation}\label{recA}
\left(\frac{\partial}{\partial t} \right)^k (s^2+t^2)^{\gamma+k} = A_k^\gamma(s,t) (s^2+t^2)^\gamma\ .
\end{equation}

\begin{proposition}
\begin{equation}\label{Gegen1}
A_k^\gamma(s,t) = c_k\left(\gamma+\frac{1}{2}\right)^{-1}(-i)^k\, s^k\,C_k^{\gamma+\frac{1}{2}}\left(\frac{t}{is} \right)\ .
\end{equation}
\end{proposition}
\begin{proof}
Let $f$ be a function of one variable, and associate the function of two variables given by $\displaystyle F(s,t) = f\left( \frac{t}{is}\right)$. Then
\[\left(\frac{\partial}{\partial t}\right)^k F(s,t) = (-i)^k s^{-k}f^{(k)}\left(\frac{t}{is}\right)
\]
Apply this relation to 
\[ f(t)= (1-t^2)^{k+\lambda-\frac{1}{2}}, \qquad F(s,t) = s^{-2k-2\lambda+1}(s^2+t^2)^{k+\lambda-\frac{1}{2}}\ .
\]
Now, by \eqref{Gegen}
\[f^{(k)}(t) = c(k, \lambda)^{-1} (1-t^2)^{\lambda-\frac{1}{2}}C_k^\lambda(t)
\]
whereas by \eqref{recA} and letting $\gamma = \lambda-\frac{1}{2}$

\[\left(\frac{\partial}{\partial t}\right)^kF(s,t) = s^{-2k-2\lambda+1}(s^2+t^2)^{\lambda-\frac{1}{2}}A_k^{\lambda-\frac{1}{2}}(s,t)\ .
\]
 \eqref{Gegen} follows.
 \end{proof}

 Notice that Theorem {\bf A3} in the introduction is merely a reformulation of Theorem \ref{theoremBC1}.

\medskip
\footnotesize{\noindent Address\\ Jean-Louis Clerc, Institut Elie Cartan, Universit\'e de Lorraine, 54506 Vand\oe uvre-l\`es-Nancy, France\medskip

\noindent \texttt{{jean-louis.clerc@univ-lorraine.fr
}}

\end{document}